\documentclass[11pt,reqno]{amsart}
\usepackage{graphicx,amsmath,amsthm}
\usepackage{bbm}
\usepackage{hyperref}
\hypersetup{colorlinks}
\pdfoutput=1

\newcommand\old[1]{}

\newcommand{\calA}{\mathcal{A}}
\newcommand{\calB}{\mathcal{B}}

\newtheorem{theorem}{Theorem}

\newtheorem{lemma}[theorem]{Lemma}
\newtheorem{corollary}[theorem]{Corollary}

\theoremstyle{remark}

\theoremstyle{definition}
\newtheorem*{definition}{Definition}

\newcommand\ghost[1]{\textcolor{white}{#1}}
\newcommand{\arxiv}[1]{{\tt \href{http://arxiv.org/abs/#1}{arXiv:#1}}}

\def\D{\mathcal{D}}
\def\SS{\mathfrak{S}}

\def\Z{\mathbb{Z}}

\def\a{{\alpha}}
\def\inv{\mathrm{inv}}
\def\blue{\textcolor{blue}}
\def\red{\textcolor{red}}

\begin{document}

\title[]{How to decompose a permutation into a pair of labeled Dyck paths by playing a game}

\author{Louis J. Billera, Lionel Levine and Karola M\'{e}sz\'{a}ros}

\address{Department of Mathematics, Cornell University, Ithaca NY 14850}
\thanks{The second author was partially supported by NSF DMS-1243606.}

\date{June 26, 2013}
 
\begin{abstract}
We give a bijection between permutations of $1,\ldots,2n$ and certain pairs of Dyck paths with labels on the down steps.  The bijection arises from a game in which two players alternate selecting from a set of $2n$ items: the permutation encodes the players' preference ordering of the items, and the Dyck paths encode the order in which items are selected under optimal play.  We enumerate permutations by certain  statistics, \emph{AA inversions} and \emph{BB inversions}, which have natural interpretations in terms of the game. We give new proofs of classical identities such as
	\[  \sum_{p}  \prod_{i=1}^n q^{h_i -1} [h_i]_q = [1]_q [3]_q \cdots [2n-1]_q
 \]	
where the sum is over all Dyck paths $p$ of length $2n$, and $h_1,\ldots,h_n$ are the heights of the down steps of $p$.
\end{abstract}

\maketitle

\section{Introduction}

Consider the following procedure applied to a permutation written in one-line notation: \emph{alternately mark ``B'' below the smallest unmarked number and then mark ``A'' below the leftmost unmarked number} until all numbers are marked.  For example, 
	\[ \begin{array}{cccccccccccc} {2}&{6}&4&1&3&{11}&5&{7}&{10}&{12}&9&8  \\ A&A&B&B&B&A&B&A&A&A&B&B \end{array} \] 
This procedure encodes the optimal play of a certain 2-player game, described below.  The purpose of this note is to analyze the procedure from a combinatorial point of view.

In the game of Ethiopian dinner \cite{LS}, Alice and Bob alternate eating morsels from a common plate.  Alice and Bob have different preferences.  We identify the set of morsels with $[2n] := \{1,\ldots,2n\}$, and fix a permutation $w$ of $[2n]$.  The players' preference orders are given by
	\begin{align*} \text{Alice } \text{prefers $j$ to $i$} & \quad \Longleftrightarrow \quad w(i)<w(j) \\
			       \text{Bob }  \text{prefers $j$ to $i$} & \quad \Longleftrightarrow \quad \ghost{w()}i<j \end{align*}
\old{
	\begin{align*} &\mbox{Alice:} && w^{-1}(1) < \ldots < w^{-1}(2n) \\ 
			       &\text{Bob:} && \ghost{w^{-1}(}1\ghost{)} < \ldots < 2n \end{align*}
That is, Alice likes morsel $w^{-1}(2n)$ the best, and Bob likes morsel $2n$ the best.}  
For concreteness, we will always assume the morsels 
are placed in increasing order from left to right, and that morsel $i$ is labeled with $w(i)$. 
Then Bob prefers morsels that are further to the right, and Alice prefers morsels with higher labels.
Assuming these preferences are common knowledge, and that Alice has the first move, which morsel should she take in order to  maximize her own enjoyment?

Kohler and Chandrasekaran \cite{KC} discovered the elegant optimal strategy, which proceeds by analyzing the moves in reverse order: On the last move, Bob eats Alice's least favorite morsel $w^{-1}(1)$. On the second to last move, Alice eats Bob's least favorite morsel from among those remaining; and so on.
In other words, the procedure described in the first paragraph assigns each morsel to the player who will eventually eat it. We call this the \emph{crossout procedure}, because on each turn a player determines what to eat by crossing out morsels that will be eaten later until only one morsel remains.

\begin{figure}[htbp] 
\begin{center} 
\includegraphics[width=\textwidth]{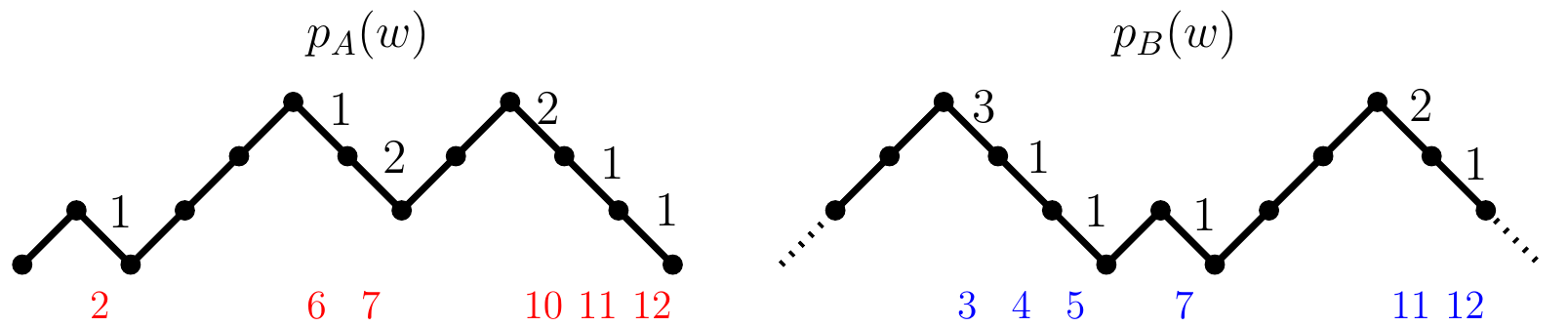} 
\smallskip
\[ w = \left[ \begin{array}{cccccccccccc} 
 1 & 2 & \mathbf{\blue{3}} & \mathbf{\blue{4}} & \mathbf{\blue{5}} & 6 & \mathbf{\blue{7}} & 8 & 9 & 10 & \mathbf{\blue{11}} & \mathbf{\blue{12}} \\
 \mathbf{\red{2}} & \mathbf{\red{6}} & 4 & 1 & 3 & \mathbf{\red{11}} & 5 & \mathbf{\red{7}} & \mathbf{\red{10}} & \mathbf{\red{12}} & 9 & 8 \\
A & A & B & B & B & A & B & A & A & A & B & B 
\end{array} \right] \]
\caption{Pair of labeled Dyck paths $p_A(w)$ and $p_B(w)$ associated to the permutation $w = 2\; 6\; 4\; 1 \;3\; 11\; 5 \;7 \;10\; 12\; 9\; 8$.  The down steps of $p_A$ correspond to the numbers marked A (shown in red) while the down steps of $p_B$ correspond to the \emph{positions} marked B (shown in blue).  The labels on the down steps of $p_A$ and $p_B$ indicate (in different ways) the relative order of the numbers marked A and B respectively.  
} 
\label{bij2n} 
\end{center} 
\end{figure}


Next we describe the \emph{crossout correspondence}, which associates to the permutation $w$ a pair of labeled Dyck paths encoding who eats what. A \emph{Dyck path} of length $2n$ is a path in $\Z^2$ from $(0,0)$ to $(2n,0)$ whose steps are either $(1,1)$ (``up'') or $(1,-1)$ (``down'') and which stays weakly above the $x$-axis.  Such a path $p$ is uniquely specified by its set of \emph{down steps} $D(p) \subset [2n]$, which is the set of $k$ such that the segment from $(k-1,h)$ to $(k,h-1)$ lies on $p$ for some $h$.  We call this $h$ the \emph{height} of the down step.  Let $\calA(w)$ and $\calB(w)$ be respectively the morsels eaten by Alice and Bob.  We let $p_A(w)$ be the path of length $2n$ with down steps $\{w(a)\}_{a\in \calA(w)}$, and let $p_B(w)$ be  the path of length $2n+2$ with down steps $\{b+1\}_{b\in \calB(w)} \cup \{2n+2\} $.

To indicate the order in which morsels are eaten, we add labels to the down steps of the paths (Figure~\ref{bij2n}).  Let $\calA(w) =\{ a_1< \ldots < a_n \}$ be the positions marked A, and let $J \in \SS_n$ be the permutation such that $w(a_i)$ is the $J(i)$th down step of $p_A$.  This down step receives label
 	\[ \ell_{J(i)} := 1 + \#\{j<i \mid w(a_j)>w(a_i)\}. \]  
For example, the down step in position~$7$ of $p_A$ in Figure~\ref{bij2n} is labeled~$2$ because there are $2-1 = 1$ numbers marked A larger than~$7$ to the left of~$7$.  

Now let $b_1,\ldots,b_n$ be the positions marked B, ordered so that $w(b_1) < \ldots < w(b_n)$.  Let $K \in \SS_n$ be the permutation such that $b_i+1$ is the $K(i)$th down step of $p_B$.  This down step receives label 
	\[ m_{K(i)} := 1 + \#\{j<i \mid b_j > b_i \}. \]
For example, the down step in position $11+1$ of $p_B$ in Figure~\ref{bij2n} is labeled~$2$ because there are $2-1 = 1$ numbers marked B smaller than $w(11)$ to the right of $w(11)$. 
Note that $p_B$ has $n+1$ down steps, the last of which has no label.

For a Dyck path $p$, denote by $h_i(p)$ be the height of the $i$th down step of~$p$, and let 
	\[ h_i^*(p) = \begin{cases} h_i(p)-1 & \text{if there is no up step to the right of the $i$th down step}  \\ 
			h_i(p) & \text{else.} \end{cases} \]

\begin{theorem} 
\label{t.bij}
The crossout correspondence is a bijection between
permutations $w \in  \mathfrak{S}_{2n}$ and 4-tuples $(p_A,p_B,(\ell_i)_{i=1}^n, (m_i)_{i=1}^n)$, such that
\smallskip
\begin{itemize}
\item $p_A$ is a Dyck path of length $2n$.
\smallskip
\item $p_B$ is a Dyck path of length $2n+2$.
\smallskip
\item $1 \leq \ell_i \leq h_i(p_A)$ for all $i=1,\ldots,n$.
\smallskip
\item $1 \leq m_i \leq h^*_i(p_B)$ for all $i=1,\ldots,n$.
\end{itemize}
\end{theorem}

We prove this theorem in \textsection\ref{s.proof} and derive some corollaries in \textsection\ref{s.cors}.  An immediate and rather surprising consequence is that if the permutation $w$ is selected uniformly at random, then the pairs $(p_A, (\ell_i)_{i=1}^n)$ and $(p_B, (m_i)_{i=1}^n)$ are independent random variables.  In other words, Alice's meal viewed in terms of Alice's preference order (i.e., her own rankings of the morsels she eats, as well as what order she eats them in) is independent of Bob's meal viewed in terms of Bob's preference order. 

\subsection*{Related work}
Theorem~\ref{t.bij} is reminiscent of the Fran\c{c}on-Viennot bijection \cite{FV} (see also \cite[\textsection5.2.15]{GJ}), which assigns to $w$ a \emph{single} labeled path of length $2n-1$, which may have East as well as Northeast and Southeast steps.  The East steps have labels between $1$ and $2h$, and the Northeast and Southeast steps have labels between $1$ and $h$, where $h$ is the height of the step. 

Hopkins and Jones \cite{HJ} pioneered the combinatorial study of Ethiopian dinner.  Defining a player's  \emph{outcome} as the set of morsels eaten by that player, they asked: As $w$ varies over all permutations of $[2n]$, how many different outcomes can arise under optimal play? They showed that for Alice the answer is the Catalan number $C_n = \frac{1}{n+1} {2n \choose n}$, and for Bob the answer is $C_{n+1}$.  For example, when $n=2$ the $2$ possible outcomes for Alice are $\{w^{-1}(2), w^{-1}(4)\}$ and $\{w^{-1}(3), w^{-1}(4)\}$; and the $5$ possible outcomes for Bob are $\{1,3\}$, $\{1,4\}$, $\{2,3\}$, $\{2,4\}$ and $\{3,4\}$.  Recalling that $C_n$ counts the number of Dyck paths of length $2n$, Theorem~\ref{t.bij} refines these results.

\section{Proof of Theorem~\texorpdfstring{\ref{t.bij}}{1}}
\label{s.proof}

\begin{proof}
We construct the inverse of the map $w \mapsto (p_A,p_B,\ell,m)$.
Given Dyck paths $p_A,p_B$ with labels $\ell,m$
on the down steps as prescribed by the statement of the theorem, construct $w \in  \mathfrak{S}_{2n}$ as follows.  

Start with $2n$ empty boxes in a line. We will write the numbers $1, \ldots, 2n$ into the boxes, one at a time, to arrive at the line notation of $w$.  Mark ``B'' below the boxes on positions belonging to the set $\{i-1 \mid i \in D(p_B)\backslash \{2n+2\}\}$, and mark ``A'' below the remaining boxes.  

Let $D(p_A)=\{a_1<a_2<\cdots <a_n\}$, and let $\ell_i$ be the label of $p_A$ on down step $a_i$.  Fill in the boxes marked $A$ one at a time as follows: For each $i=1,\ldots,n$, write $a_i$ in the $\ell_i$th empty box marked $A$, counting from the left.  This is always possible because $\ell_i \leq h_i$ and at most $n-h_i$ boxes marked A have been filled, so there are at least $h_i$ empty boxes marked A remaining.

Now we fill the boxes marked $B$ with the numbers in $S := [2n] - D(p_A)$. Let $S=\{b_1<b_2<\cdots <b_n\}$, and let $m_i$ be the label of $p_B$ on its $i$th down step. First write $b_{m_1}$ in the leftmost box marked B.  Delete $b_{m_1}$ from $S$ and re-index by $1,\ldots,n-1$ to obtain a new set $S=\{b_1<b_2<\cdots <b_{n-1}\}$.
Likewise, for each $2 \leq i \leq n$, write $b_{m_i}$ in the leftmost empty box marked B.  Delete $b_{m_i}$ from $S$ and re-index to obtain a new set $S=\{b_1<b_2<\cdots <b_{n-i}\}$.  This is always possible because $m_i \leq h_i^* \leq n+1-i$.

After these $n$ steps are completed, we have filled all $2n$ boxes to obtain the line notation of the permutation $w$.  
\end{proof}

\section{\texorpdfstring{$q$}{q}-analogues and marginals}
\label{s.cors}

Write $\D_{2n}$ for the set of all Dyck paths of length $2n$.  Given $\alpha \in \D_{2n}$ and $\beta \in \D_{2n+2}$, denote by $\SS(\alpha,\beta)$ the set of all $w \in \SS_{2n}$ such that $p_A(w)=\alpha$ and $p_B(w)=\beta$.  By Theorem~\ref{t.bij} we have $\# \SS(\alpha,\beta) = \prod_{i=1}^n h_i(\alpha)h_i^*(\beta)$.  We can refine this enumeration by keeping track of certain inversions in $w$.

\begin{definition}
For $X,Y \in \{A,B\}$, an \emph{XY inversion} of $w \in \SS_{2n}$ is a pair of indices $i<j$ such that $w(i)>w(j)$ and $i$ is marked $X$ and $j$ is marked $Y$. 
\end{definition}

Each AA inversion represents a pair of morsels that Alice eats ``out of order'' (i.e., she eats the less preferred morsel first) and each BB inversion represents a pair of morsels that Bob eats out of order. These statistics were not among the over $9800$ permutation statistics indexed by \url{findstat.org} \cite{findstat} as of June 2013.

Denote by $aa(w)$ and $bb(w)$ respectively the number of AA inversions and BB inversions in $w$.  Let $q$ and $r$ be formal variables, and for integer $h$ write $[h]_q := 1+q+\cdots + q^{h-1}$.

\begin{theorem}
\[ \sum_{w \in \SS(\alpha,\beta)} q^{aa(w)} r^{bb(w)} = \prod_{i=1}^n [h_i(\alpha)]_q [h_i^*(\beta)]_r. \]
\end{theorem}

\begin{proof}
When we write $a_i$ in the $\ell_i$th empty box marked A, the $\ell_i-1$ empty boxes marked A to its left will later receive larger numbers, so the total number of AA inversions created is $\sum_{i=1}^n (\ell_i-1)$.

When we write $b_{m_i}$ in the leftmost empty box marked B, the smaller numbers $b_{m_1},\ldots,b_{m_{i-1}}$ will later be placed in boxes marked B to its right, so the total number of BB inversions created is $\sum_{i=1}^n (m_i -1)$.  Summing over all labelings $\ell$ and $m$ satisfying the height restrictions $1 \leq \ell_i \leq h_i(\alpha)$ and $1 \leq m_i \leq h_i^*(\beta)$ for $i=1,\ldots,n$, we obtain

	\begin{align*} \sum_{w \in \SS(\alpha,\beta)} q^{aa(w)} r^{bb(w)} &= \sum_{\ell,m} q^{\ell_1-1} \cdots q^{\ell_n-1} r^{m_1-1} \cdots r^{m_n-1} \\ &= [h_1(\alpha)]_q \cdots [h_n(\alpha)]_q [h_1^*(\beta)]_r \cdots [h_n^*(\beta)]_r. \qed
	\end{align*}
\renewcommand{\qedsymbol}{}
\end{proof}

\begin{theorem}[Alice's Marginal]
\label{t.alice} 
Fix a Dyck path $\alpha \in \D_{2n}$. The number of permutations $w \in \SS_{2n}$ such that $p_A(w)=\alpha$ is \begin{equation} \label{cute} 2\cdot 4 \cdot 6 \cdot \ldots \cdot (2n) \prod_{i=1}^n h_i(\alpha).\end{equation}
\end{theorem}

\begin{proof} Given a Dyck path $\alpha$,
we describe a process which generates all permutations $w \in \SS_{2n}$ such that $p_A(w)=\alpha$.
Start with $2n$ empty boxes in a line. We write the numbers $1, \ldots, 2n$ into the boxes, one number at a time, to arrive at the line notation of a permutation $w$.  This write-in process  proceeds in $2n$ rounds.  Round $k$ consists of the following:

\emph{If the $k$th step of $\alpha$ is a down step}: Write $k$ in an empty circled box.

\emph{If the $k$th step of $\alpha$ is an up step}: Write $k$ in an empty uncircled box, and then circle the leftmost empty uncircled box.

In round $1$, we pick one of the $2n$ empty boxes and write $1$ in it. Circle the leftmost box that is empty and is not circled. This is the end of round $1$.

At the beginning of round $k$, the numbers $1, 2, \ldots, k-1$ are already written in some of the $2n$ boxes, and a total of $u$ boxes are circled, where $u$ is the number of up steps among the first $k-1$ steps of $\alpha$.  Among these circled boxes, $d$ of them have numbers written in them already, where $d = k-1-u$ is the number of down steps among the first $k-1$ steps of $\alpha$. 

If the $k$th step of $\alpha$ is a down step, then the number $k$ must be written in one of the empty circled boxes.  There are $u-d = h_{d+1}(\alpha)$ such boxes.

If the $k$th step of $\alpha$ is an up step, then the number $k$ must be written in one of the empty uncircled boxes.  There are $2n - u - (k-1) + d = 2n - 2u$ such boxes.
 
 The above process yields all 
  $w \in \mathfrak{S}_{2n}$ such that $p_A(w)=\alpha$.
For the $(d+1)$st down step we had $h_{d+1}(\alpha)$ choices, and for the $(u+1)$st up step we had $2n-2u$ choices, so the total number of choices is given by (\ref{cute}).
\end{proof}

An analogous argument (this time filling in the $2n$ boxes from left to right) proves the following.

\begin{theorem}[Bob's Marginal]
\label{t.bob}
Fix a Dyck path $\beta \in \D_{2n+2}$. The number of permutations $w \in \SS_{2n}$ such that $p_B(w)=\beta$ is 
	\[ 1\cdot 3 \cdot 5 \cdot \ldots \cdot (2n-1) \prod_{i=1}^n h_i^*(\beta).\]
\end{theorem}

\begin{corollary}
	\[ \sum_{\alpha \in \D_{2n}} \prod_{i=1}^n h_i(\alpha) = 1 \cdot 3 \cdot \ldots \cdot (2n-1) \]
	\[ \sum_{\beta \in \D_{2n+2}} \prod_{i=1}^n h_i^*(\beta) = 2 \cdot 4 \cdot \ldots \cdot 2n. \]\end{corollary}

\begin{proof}
The first identity follows from Theorem~\ref{t.alice} by summing over all Dyck paths $\alpha \in \D_{2n}$, the second from Theorem~\ref{t.bob} by summing over all $\beta \in \D_{2n+2}$. In both cases the sum is $(2n)!$ because each $w\in \mathfrak{S}_{2n}$ is counted once.
\end{proof}

We can refine Theorem~\ref{t.alice} to count AA inversions along with another statistic 
	\[ z(w) := \# \{ i<j \mid w(i)<w(j) \mbox{ and } i \in \mathcal{B}(w) \}. \]

\begin{theorem} \label{qt}
For any $\alpha \in \D_{2n}$, we have
	 \begin{equation} \label{eq:qt} \sum_{w \in \SS_{2n} \; : \; p_A(w)=\alpha} q^{aa(w)} t^{z(w)} =  [2]_t  [4]_t  [6]_t \cdots [2n]_t \prod_{i=1}^n [h_i(\alpha)]_q.  \end{equation} 
\end{theorem}

\begin{proof}
The circled boxes in the proof of Theorem~\ref{t.alice} are those marked A.  On the $(d+1)$st down step of $\alpha$, we write a number in one of the $h_{d+1}$ empty circled boxes, which creates an AA inversion with each empty circled box to its left.

The uncircled boxes are those marked B.  On the $(u+1)$st up step of $\alpha$, we write a number in one of the $2n-2u$ empty uncircled boxes, which creates a non-inversion with each empty box to its right. (Note that all boxes to its right are currently uncircled, although some may be circled later.)  \end{proof}

Next we derive from Theorem~\ref{qt} the identity mentioned in the abstract.  The idea will be to express $z(w)$ in terms of AB and BB inversions, and then set $q=t^{-1}$.  First we make the observation that \emph{there are no BA inversions}, i.e., $ba(w)=0$ for all $w \in \SS_{2n}$.  Indeed, suppose that $i \in \mathcal{B}(w)$ and $j \in \mathcal{A}(w)$ and $i<j$.  When $w(j)$ is marked A it is the leftmost unmarked number, so $w(i)$ must already have been marked B.  When $w(i)$ was marked B it was the smallest unmarked number, so $w(i)<w(j)$, which shows that the pair $(i,j)$ is not an inversion.

The equation $ba(w)=0$ says that at the end of the game, there is no pair of morsels that Alice and Bob would agree to trade: indeed, $i<j$ is such a pair if and only if $i$ is marked B and $j$ is marked A and $w(i)>w(j)$ (so that Bob prefers $j$ to $i$ but received $i$, and Alice prefers $i$ to $j$ but received $j$).  
The lack of such pairs is a consequence of the Pareto efficiency of the crossout strategy, proved by Brams and Straffin \cite[Theorem~1]{BS}.

\begin{lemma}
\begin{equation} \label{eq:z} z(w) = n^2+\sum_{i=1}^n (h_i(p_A(w))-1)-ab(w)-bb(w).\end{equation} \end{lemma}
 \begin{proof} By the definition of $z(w)$, 
\[ z(w)=\# \{ i<j \mid  i \in \mathcal{B}(w)\}-\# \{ i<j \mid w(i)>w(j) \mbox{ and } i \in \mathcal{B}(w) \}.\] Note that \[\# \{ i<j \mid  i \in \mathcal{B}(w)\}=n^2+\sum_{i \in D(\a)}(h_i(p_A(w))-1)-ab(w)\] and \[\# \{ i<j \mid w(i)>w(j) \mbox{ and } i \in \mathcal{B}(w) \}=bb(w)+ba(w),\] yielding \eqref{eq:z} since $ba(w)=0$.
\end{proof}

\begin{corollary}\label{cor:a} \begin{equation} \label{eq:a} \sum_{\alpha \in \mathcal{D}_{2n}}  \prod_{i=1}^n q^{h_i(\a) -1} [h_i(\a)]_q = [1]_q [3]_q \cdots [2n-1]_q. \end{equation} \end{corollary}

\begin{proof} Let \[ \inv(w) = aa(w) + ab(w) + bb(w) \] be the total number of inversions of $w$.  Setting $t=q^{-1}$ in (\ref{eq:qt}) yields 

 \begin{equation*} \sum_{w \in \SS_{2n} \; : \; p_A(w)=\alpha} q^{\inv(w)-n^2-\sum_{i=1}^n (h_i(\a)-1) }=  [2]_{q^{-1}} [4]_{q^{-1}} \cdots [2n]_{q^{-1}} \prod_{i=1}^n [h_i(\alpha)]_q.  \end{equation*} 
 
 Now using the fact that $[k]_{q^{-1}} = q^{1-k} [k]_q$, we obtain
 
 \begin{equation} \label{eq:qq} \sum_{w \in \SS_{2n} \; : \; p_A(w)=\alpha} q^{\inv(w)}=  [2]_{q} [4]_q \cdots [2n]_{q} \prod_{i=1}^n q^{h_i(\alpha)-1}[h_i(\alpha)]_q.  \end{equation}
 
Summing (\ref{eq:qq}) over $\alpha \in \mathcal{D}_{2n}$ we count every permutation $w \in \SS_{2n}$ once, hence 

 \begin{equation*}  \sum_{w \in \SS_{2n}} q^{\inv(w)}=  [2]_{q} [4]_q  \cdots [2n]_{q} \sum_{\alpha \in \mathcal{D}_{2n}} \prod_{i=1}^n q^{h_i(\alpha)-1}[h_i(\alpha)]_q,  \end{equation*} 
which, together with the well-known identity  
 
 \begin{equation*}  \sum_{w \in \SS_{2n}} q^{\inv(w)}=   [1]_{q} [2]_{q}  \cdot \ldots \cdot [2n]_{q},   \end{equation*} 
yields \eqref{eq:a}.
\end{proof}

\section{Random preferences}

Suppose Alice's preferences are random: $w$ is chosen uniformly among all permutations in $\SS_{2n}$.  Under the crossout procedure, how likely is Alice to eat $w^{-1}(k)$, the morsel she ranks as the $k$th worst?  We will show that the answer is $\frac{k-1}{2n-1}$; in particular, she always eats her favorite morsel and never gets stuck with her least favorite.

A Dyck path with each down step labeled by an integer between $1$ and its height is sometimes called a \emph{Hermite history} \cite{Vie}.  
The proof of the next theorem uses a well-known bijection between Hermite histories and matchings of $[2n]$, which proceeds from left to right matching each down step of the Dyck path with a previous up step. If the down step is labeled $L$, then we match it with the $L$th up step that is not yet matched.

\begin{theorem}
Fix $1 \leq m \leq n$ and $1 \leq k_1 < \cdots < k_m \leq 2n$.
If $w \in \SS_{2n}$ is chosen uniformly at random, then the probability that Alice eats all of the morsels $w^{-1}(k_1),\ldots,w^{-1}(k_m)$ is
	\[ \prod_{i=1}^m \frac{k_i - 2i + 1}{2n-2i+1}. \]
\end{theorem}

\begin{proof}
According to Theorem~\ref{t.bij}, if $w \in \SS_{2n}$ is a uniform random permutation, then $(p_A, \ell)$ is a uniform random Hermite history of length $2n$.
We must therefore count the number of Hermite histories of length $2n$ in which $k_1,\ldots,k_\ell$ are all down steps.  As described above, these are in bijection with matchings of $[2n]$ in which each $k_i$ is matched to a number less than $k_i$.  The number of such matchings is 
\[ (k_1-1)(k_2-3)\cdots(k_\ell-2\ell+1) \cdot (2n-2\ell-1)(2n-2\ell-3) \cdots 1. \]  Dividing by the total number of matchings $(2n-1)(2n-3)\cdots 1$ yields the result.
\end{proof}

\section{Games of odd length}

Following \cite{LS}, we adopt the convention that Bob always has the last move.  Thus if the number of morsels is odd, say $2n-1$, then Bob will move both first and last.  As a result, the combinatorial description of the crossout procedure (alternately mark B below the smallest unmarked number and then mark A below the leftmost unmarked number) is unchanged.  Let $\calA(w)$ be the set of $n-1$ morsels eaten by Alice, and $\calB(w)$ the complementary set of $n$ morsels eaten by Bob.  We modify the construction of the Dyck paths as follows: $p_A$ is the Dyck path of length $2n$ with down steps $\{w(a)\}_{a \in \calA(w)} \cup \{2n\}$, and $p_B$ is the Dyck path of length $2n$ with down steps $\{b+1\}_{b \in \calB(w)}$.  In other words, we append a final down step to Alice's path and an initial up step to Bob's path.  We label the down steps of each path as before; this time, since only $n-1$ positions are marked A, the last down step of $p_A$ has no label.

\begin{theorem}
\label{t.odd}
The crossout correspondence gives a bijection between permutations $w \in  \mathfrak{S}_{2n-1}$ and 4-tuples $(p_A,p_B,(\ell_i)_{i=1}^{n-1}, (m_i)_{i=1}^n)$, such that
\smallskip
\begin{itemize}
\item $p_A$ and $p_B$ are Dyck paths of length $2n$.
\smallskip
\item $1 \leq \ell_i \leq h_i^*(p_A)$ for all $i=1,\ldots,n-1$.
\smallskip
\item $1 \leq m_i \leq h_i(p_B)$ for all $i=1,\ldots,n$.
\end{itemize}
\end{theorem}

\old{
\section{more to do}

\begin{itemize}
\item isn't there a shorter proof of Theorem 3 using Theorem 1? just observe that the total number of Bob outcomes (which morsels he eats and what order he eats them in) is 2*4*...*2n.  On the other hand the current proof is shows the refinement, Theorem 6.
\item AB inversions?
\item ref Flajolet continued fractions \cite{Fla}?
\item ref for bijection between Hermite histories and matchings? Viennot?
\end{itemize}
}

\section*{Acknowledgements}
We thank Alex Postnikov and Richard Stanley for helpful discussions.

\end{document}